\documentclass[11pt,reqno]{amsart}
\usepackage[colorlinks]{hyperref}
\usepackage[all]{xy}
\usepackage{amssymb}

\title{Twisted Fourier--Mukai partners \linebreak of Enriques surfaces}

\author[N.~Addington]{Nicolas Addington}
\address{Nicolas Addington \\
Department of Mathematics \\
University of Oregon \\
Eugene, Oregon 97403 \\
United States}
\email{adding@uoregon.edu}

\author[A.~Wray]{Andrew Wray}
\address{Andrew Wray \\
Department of Mathematics \\
University of Oregon \\
Eugene, Oregon 97403 \\
United States}
\email{awray3@uoregon.edu}

\renewcommand \O {\mathcal O}
\newcommand \Q {\mathbb Q}
\newcommand \Z {\mathbb Z}

\renewcommand \phi \varphi

\DeclareMathOperator \Br {Br}
\DeclareMathOperator \Sq {Sq}
\DeclareMathOperator \Hom {Hom}
\DeclareMathOperator \Pic {Pic}

\newtheorem* {thm*} {Theorem}
\newtheorem {prop} {Proposition} [section]
\theoremstyle{definition}
\newtheorem {rmk} [prop] {Remark}

\begin{document}

\begin{abstract}
Bridgeland and Maciocia showed that a complex Enriques surface $X$ has no Fourier--Mukai partners apart from itself: that is, if $D^b(X) \cong D^b(Y)$ then $X \cong Y$.  We extend this to twisted Fourier--Mukai partners: if $\alpha$ is the non-trivial element of $\Br(X) = \Z/2$ and $D^b(X,\alpha) \cong D^b(Y,\beta)$, then $X \cong Y$ and $\beta$ is non-trivial.  Our main tools are twisted topological K-theory and twisted Mukai lattices.
\end{abstract}

\maketitle

\section*{Introduction}

Two smooth projective varieties $X$ and $Y$ are called \emph{Fourier--Mukai partners} if they have equivalent derived categories of coherent sheaves $D^b(X) \cong D^b(Y)$, and \emph{twisted} Fourier--Mukai partners if they have equivalent derived categories of twisted sheaves $D^b(X,\alpha) \cong D^b(Y,\beta)$ for some Brauer classes $\alpha \in \operatorname{Br}(X)$ and $\beta \in \operatorname{Br}(Y)$.  In the early 2000s, Bridgeland, Maciocia, and Kawamata showed that among complex surfaces, only K3, abelian, and elliptic surfaces have non-trivial Fourier--Mukai partners; see \cite[Ch.~12]{huybrechts_fm} for a textbook account.  Recently several authors have been interested in extending this result to positive characteristic and to twisted Fourier--Mukai partners.  Here we carry out one step in this program:

\begin{thm*}
Let $X$ be a complex Enriques surface, let $\alpha \in \Br(X) = \Z/2$, and let $Y$ be another smooth complex projective variety and $\beta \in \Br(Y)$.  If $D^b(X,\alpha) \cong D^b(Y,\beta)$, then $X \cong Y$, and via that isomorphism $\alpha = \beta$.
\end{thm*}

\noindent If $\alpha$ and $\beta$ are trivial then this was proved by Bridgeland and Maciocia in \cite[Prop.~6.1]{bm2}.  Some special cases of the twisted result were obtained by Martinez Navas in \cite[Ch.~3]{mn}.  

Dimension, order of the canonical bundle, and Hochschild homology are invariant under twisted derived equivalence, just as they are under untwisted equivalence: the proofs of \cite[Prop.~4.1 and Rem.~6.3]{huybrechts_fm} go through unchanged, relying on existence and especially uniqueness of kernels for twisted equivalences due to Canonaco and Stellari \cite[Thm.~1.1]{cs1}.  Thus $Y$ is an Enriques surface.  In \S\ref{ktop} we use twisted topological K-theory to show that we cannot have $\alpha$ trivial and $\beta$ non-trivial.  In \S\ref{lattice} we show that if $\alpha$ and $\beta$ are both non-trivial then $X \cong Y$; our proof follows the outline of Bridgeland and Maciocia's, but is more delicate.

\subsection*{Acknowledgements}
This paper grew out of conversations with K.~Honigs and S.~Tirabassi, related to \cite{hlt}.  We thank them for a fruitful exchange.  We also thank A.~Beauville and D.~Dugger for their advice, and the referee for helpful suggestions.  Our first attempts at Proposition \ref{the_calc} made heavy use of Macaulay2 \cite{M2}; we thank B.~Young for computer time.

\section{One twisted, one untwisted} \label{ktop}
Given a smooth complex projective variety $X$ and a class $\alpha \in \Br(X)$ with image $\bar\alpha \in H^3(X,\Z)$, we let $K^i_\text{top}(X,\bar\alpha)$ denote twisted topological K-theory; for the definition and first properties we refer to Atiyah and Segal \cite{as1,as2}.  It is a 2-periodic sequence of finitely generated Abelian groups, and can be computed using an Atiyah-Hirzebruch spectral sequence.

An untwisted derived equivalence induces an isomorphism on topological K-theory, and recent work of Moulinos \cite[Cor.~1.2]{tasos},\footnote{This reference is very $\infty$-categorical; a more down-to-earth reader might want to say that for any kernel $P \in D^b(X \times Y,\alpha^{-1} \boxtimes \beta)$, the class $[P] \in K^0_\text{top}(X \times Y, \bar\alpha^{-1} \boxtimes \bar\beta)$ induces a map $K^i_\text{top}(X,\bar\alpha) \to K^i_\text{top}(Y,\bar\beta)$ in a way that's functorial with respect to composition of kernels.  But this would require a compatibility between pushforward on algebraic and topological twisted K-theory, comparable to \cite{ah2}.  This seems to be missing from the literature, and to prove it here would take us too far afield.} together with uniqueness of dg enhancements \cite[\S6.3]{cs2}, extends this to the twisted case: if $D^b(X,\alpha) \cong D^b(Y,\beta)$ then $K^i_\text{top}(X,\bar\alpha) \cong K^i_\text{top}(Y,\bar\beta)$.

We will show that if $X$ is an Enriques surface then $K^1_\text{top}(X) = \Z/2$, but if $\alpha$ is the non-trivial element of $\Br(X) = \Z/2$ then $K^1_\text{top}(X,\bar\alpha) = 0$, and thus an untwisted Enriques surface cannot be derived equivalent to a twisted one.\medskip

By \cite[Lem.~VIII.15.1]{barth}, an Enriques surface has $\pi_1 = \Z/2$ and Hodge diamond
\[ \begin{smallmatrix}
& & 1 \\
& 0 & & 0 \\
0 & & 10 & & 0 \\
& 0 & & 0 \\
& & \phantom{.}1.
\end{smallmatrix} \]
By the universal coefficient theorem and Poincar\'e duality, it follows that
\[ H^i(X,\Z) = \begin{cases} 
\Z & i = 0 \\
0 & i = 1 \\
\Z^{10} \oplus \Z/2 & i = 2 \\
\Z/2 & i = 3 \\
\Z & i = 4.
\end{cases} \]
Now the claims about $K^1_\text{top}$ above follow from:

\begin{prop}
If $X$ is any compact complex surface, then
\[ K^1_\text{top}(X) \cong H^1(X,\Z) \oplus H^3(X,\Z). \]
If $\alpha \in \Br(X)$ has image $\bar\alpha \in H^3(X,\Z)$, then
\[ K^1_\text{top}(X,\bar\alpha) \cong H^1(X,\Z) \oplus H^3(X,\Z)/\bar\alpha. \]
\end{prop}

\begin{proof}
We abbreviate $H^i(X,\Z)$ as $H^i$.  The $E_3$ page of the Atiyah--Hirzebruch spectral sequence is
\[ \xymatrix@R=3pt@C=20pt{
\vdots \ar[rrrdd] & \vdots \ar[rrrdd] & \vdots & \vdots & \vdots \\
0 & 0 & 0 & 0 & 0 \\
H^0 \ar[rrrdd] & H^1 \ar[rrrdd] & H^2 & H^3 & H^4 \\
0 & 0 & 0 & 0 & 0 \\
H^0 \ar[rrrdd] & H^1 \ar[rrrdd] & H^2 & H^3 & H^4 \\
0 & 0 & 0 & 0 & 0 \\
\vdots & \vdots & \vdots & \vdots & \vdots \\
} \]
For untwisted K-theory, the map $d_3$ is given by
\[ \Sq^3_\Z = \beta \circ {\Sq^2} \circ r, \]
where $r$ is reduction mod 2, $\Sq^2$ is the usual Steenrod square, and $\beta$ is the Bockstein homomorphism associated to the short exact sequence of coefficient groups
\[ \xymatrix{
0 \ar[r] & \Z \ar[r]^2 & \Z \ar[r]^-r & \Z/2 \ar[r] & 0.
} \]
This vanishes on $H^0$ and $H^1$ for degree reasons, so the spectral sequence degenerates.  The filtration of $K^1_\text{top}(X)$ splits because $H^1$ is free.

For twisted K-theory, the $E_3$ page has the same terms, but now
\[ d_3(x) = \Sq^3_\Z(x) - \bar\alpha \cup x \]
by \cite[Prop.~4.6]{as2}.  This maps $1 \in H^0$ to $-\bar\alpha \in H^3$, and vanishes on $H^1$ because $\bar\alpha$ is torsion and $H^4$ is free.  Thus the $E_4$ page is
\[ \xymatrix@R=3pt@C=20pt{
\vdots & \vdots & \vdots & \vdots & \vdots \\
0 & 0 & 0 & 0 & 0 \\
k \cdot H^0 & H^1 & H^2 & H^3/\bar\alpha & H^4 \\
0 & 0 & 0 & 0 & 0 \\
k \cdot H^0 & H^1 & H^2 & H^3/\bar\alpha & H^4 \\
0 & 0 & 0 & 0 & 0 \\
\vdots & \vdots & \vdots & \vdots & \vdots \\
} \]
where $k$ is the order of $\bar\alpha$.  At this point the spectral sequence degenerates, and again the filtration of $K^1_\text{top}(X,\alpha)$ splits because $H^1$ is free.
\end{proof}

\section{Both twisted} \label{lattice}

We begin by recalling the outline of Bridgeland and Maciocia's proof that if $X$ and $Y$ are complex Enriques surfaces and $D^b(X) \cong D^b(Y)$, then $X \cong Y$; see \cite[Prop.~6.1]{bm2} for the original or \cite[Prop.~12.20]{huybrechts_fm} for another account.  Take universal covers $p\colon \tilde X \to X$ and $q\colon \tilde Y \to Y$, so $\tilde X$ and $\tilde Y$ are K3 surfaces, and let $\tau$ denote the covering involution of either $\tilde X$ or $\tilde Y$.  An equivalence $D^b(X) \to D^b(Y)$ lifts to an equivalence $D^b(\tilde X) \to D^b(\tilde Y)$ that commutes with $\tau^*$.  The induced Hodge isometry $H^*(\tilde X, \Z) \to H^*(\tilde Y, \Z)$ commutes with $\tau^*$, and hence restricts to a Hodge isometry between the $\tau^*$-anti-invariant parts $H^2_-(\tilde X, \Z) \to H^2_-(\tilde Y, \Z)$.  Using Nikulin's lattice theory, this extends to a Hodge isometry on all of $H^2$, still commuting with $\tau^*$.  Thus $X \cong Y$ by the Torelli theorem for Enriques surfaces.

Now let $\alpha \in \Br(X)$ and $\beta \in \Br(Y)$ be non-trivial.  First we will check that an equivalence $D^b(X,\alpha) \to D^b(Y,\beta)$ lifts to an equivalence $D^b(\tilde X, p^*\alpha) \to D^b(\tilde Y, q^* \beta)$ that commutes with $\tau^*$.  Next we will make a careful choice of B-fields in $H^2(\tilde X, \Q)$ and $H^2(\tilde Y, \Q)$ that lift $p^* \alpha$ and $q^* \beta$ and satisfy $\tau^* B = -B$.  Then we have an induced isometry $\phi\colon H^*(\tilde X,\Z) \to H^*(\tilde Y,\Z)$, such that $e^{-B} \circ \phi \circ e^B$ preserves $H^{2,0}$ and commutes with $\tau^*$.  This yields a Hodge isometry $H^2_-(\tilde X, \Q) \to H^2_-(\tilde Y, \Q)$, and the delicate step is to show that it takes $H^2_-(\tilde X, \Z)$ into $H^2_-(\tilde Y, \Z)$.  Then we can conclude as in the untwisted case.

\subsection{Lifting the kernel}

\begin{prop} \label{equivariant}
With the notation introduced above, if $D^b(X,\alpha) \cong D^b(Y,\beta)$ then there is a kernel $\tilde P \in D^b(\tilde X \times \tilde Y, p^* \alpha^{-1} \boxtimes q^* \beta)$ that induces an equivalence $D^b(\tilde X, p^* \alpha) \to D^b(\tilde Y, q^*\beta)$ and satisfies $(\tau \times \tau)^* \tilde P \cong \tilde P$.
\end{prop}

\begin{rmk}
The expert reader might worry that $(\tau \times \tau)^* \tilde P$ lies \emph{a priori} in $D^b(\tilde X \times \tilde Y, \tau^* p^* \alpha^{-1} \boxtimes \tau^* q^* \beta)$, and that in order to identify this with $D^b(\tilde X \times \tilde Y, p^* \alpha^{-1} \boxtimes q^* \beta)$ we might have to make some non-canonical choice; cf.\ \cite[Rmk.~1.2.9]{andrei}.  But once we fix a cocycle $\{ U_i, \alpha_{ijk} \}$ representing $\alpha$, the cocycle $\{ p^{-1}(U_i), \alpha_{ijk} \circ p \}$ representing $p^* \alpha$ is actually the same as the cocycle $\{ \tau^{-1}(p^{-1}(U_i)), \alpha_{ijk} \circ p \circ \tau \}$ representing $\tau^* p^* \alpha$, not just cohomologous, because $p \circ \tau = p$.  The same is true of $\beta$.  So the identification is canonical.
\end{rmk}

\begin{proof}[Proof of Proposition \ref{equivariant}]
By \cite[Thm.~1.1]{cs1}, the equivalence is induced by a kernel $P \in D^b(X \times Y, \alpha^{-1} \boxtimes \beta)$.  To lift it to a kernel $\tilde P$ as in the statement of the proposition, we can follow Bridgeland and Maciocia \cite[Thm.~4.5]{bm1}, or Huybrechts' book \cite[Prop.~7.18]{huybrechts_fm}, or Lombardi and Popa \cite[Thm.~10]{lp} 
with no changes.  The key point is an equivalence between:
\begin{enumerate}
\item $(p^* \alpha^{-1} \boxtimes q^* \beta)$-twisted sheaves on $\tilde X \times \tilde Y$,
\medskip

\pagebreak

\item $(p^* \alpha^{-1} \boxtimes \beta)$-twisted sheaves on $\tilde X \times Y$ that are modules over
\[ (1 \times q)_* \O_{\tilde X \times \tilde Y} = \O_{\tilde X \times Y} \oplus \omega_{\tilde X \times Y}, \]
and
\medskip

\item $(p^* \alpha^{-1} \boxtimes \beta)$-twisted sheaves $F$ on $\tilde X \times Y$ with $F \otimes \omega_{\tilde X \times Y} \cong F$.
\end{enumerate}
In fact there is a subtlety in identifying (2) and (3), which the references above elide, but which Krug and Sosna treat carefully in \cite[Lem.~3.6(ii)]{ks}. 
To turn a sheaf as in (3) into a $(\O \oplus \omega)$-module as in (2), one needs the chain of isomorphisms $F \otimes \omega^2 \cong F \otimes \omega \cong F$ to agree with the global identification $\omega^2 \cong \O$.  But in our case, the complex $(p \times 1)^* P$ that we wish to lift is simple, so any discrepancy can be scaled away before we start lifting cohomology sheaves.

To see that $(p \times 1)^* P$ is simple, first observe that it is the composition of the kernels $P \in D^b(X \times Y)$ and $\O_{\Gamma_p} \in D^b(\tilde X \times X)$, where $\Gamma_p$ is the graph of $p$, by \cite[Ex.~5.12 and 5.4(ii)]{huybrechts_fm}.   Moreover, because $P$ induces an equivalence, composition with $P$ is an equivalence $D^b(\tilde X \times X) \to D^b(\tilde X \times Y)$, so
\[ \Hom_{\tilde X \times Y}((p \times 1)^* P, (p \times 1)^* P)
= \Hom_{\tilde X \times X}(\O_{\Gamma_p}, \O_{\Gamma_p})
= H^0(\O_{\Gamma_p}). \]
This is 1-dimensional because $\Gamma_p \cong \tilde X$.
\end{proof}

\subsection{Choice of B-field}
To get induced maps on cohomology from our kernel $\tilde P$, we must choose B-field lifts of our Brauer classes, that is, a class $B \in H^2(\tilde X, \Q)$ with $\exp(B^{0,2}) = p^*\alpha$, and similarly with $q^* \beta$.

By \cite[Lem.~VIII.19.1]{barth}, we can choose an isometry
\begin{equation} \label{good_basis}
H^2(\tilde X,\Z) \cong -E_8 \oplus -E_8 \oplus U \oplus U \oplus U
\end{equation}
under which the involution $\tau^*$ acts as
\begin{equation} \label{involution}
(x,y,z_1,z_2,z_3) \mapsto (y,x,z_2,z_1,-z_3).
\end{equation}
Here $-E_8$ is the unique negative definite even unimodular lattice of rank 8, and $U$ is the standard hyperbolic lattice, with basis $e$ and $f$ satisfying $e^2 = f^2 = 0$ and $e.f = 1$.

The following is essentially due to Beauville \cite{beauville}, as we explain in the proof.

\begin{prop} \label{beauville_prop}
Under the isometry \eqref{good_basis}, the class
\[ B := (0,0,0,0,\tfrac12 e + \tfrac12 f) \in H^2(\tilde X, \Q) \]
satisfies $\exp(B^{0,2}) = p^* \alpha \in \Br(\tilde X) \subset H^2(\O_{\tilde X}^*)$.
\end{prop}

\begin{rmk}
Note that $p^* \alpha$ may be trivial: it may be that $B$ has the same $(0,2)$ part as some integral class in $H^2(\tilde X,\Z)$.  Indeed, the point of Beauville's beautiful paper is that the set of Enriques surfaces for which this happens form a countable union of divisors in the moduli space.
\end{rmk}

\pagebreak

\begin{proof}[Proof of Proposition \ref{beauville_prop}]
Consider the diagram of sheaves
\[ \xymatrix@C=35pt{
0 \ar[r] & \Z \ar[d]^r \ar[r]^{2 \pi i} & \O \ar[d]^{\exp(\tfrac12 -)} \ar[r]^\exp & \O^* \ar@{=}[d] \ar[r] & 0 \\
0 \ar[r] & \Z/2 \ar[r] & \O^* \ar[r]_{z \mapsto z^2} & \O^* \ar[r] & 0 \\
} \]
on either $X$ or $\tilde X$.  On the Enriques surface $X$, we have $H^{0,1} = H^{0,2} = 0$, so $\Br(X) = H^2(\O_X^*) = H^3(X,\Z) = \Z/2$, and taking cohomology we get
\[ \xymatrix{
0 \ar[d] \ar[r] & \Pic(X) \ar@{=}[d] \ar[r]^{c_1} & H^2(X,\Z) \ar[d]^r \ar[r] & 0 \ar[d] \\
\Pic(X) \ar[r]^2 & \Pic(X) \ar[r] & H^2(X,\Z/2) \ar[r] & \Br(X) \ar[r]^0 & \Br(X).
} \]
On the K3 surface $\tilde X$, we get
\[ \xymatrix{
0 \ar[d] \ar[r] & \Pic(\tilde X) \ar@{=}[d] \ar[r]^{c_1} & H^2(\tilde X,\Z) \ar[d]^r \ar[r] & H^{0,2}(\tilde X) \ar[d]^{\exp(\tfrac12-)} \\
\Pic(\tilde X) \ar[r]^2 & \Pic(\tilde X) \ar[r] & H^2(\tilde X,\Z/2) \ar[r] & H^2(\O_{\tilde X}^*)
} \]
Moreover the pullback $p^*$ maps the first diagram to the second.

In the second diagram, consider $2B \in H^2(\tilde X, \Z)$.  The reduction $r(2B)$ is called $\varepsilon$ in the notation of \cite[\S5]{beauville}, and by [ibid., Prop.~5.3] we can choose $x \in H^2(X,\Z/2)$ with $p^* x = r(2B)$.  By [ibid., Lem.~5.4], we have $x^2 = 1$, so $x$ is not the reduction of an integral class $y \in H^2(X,\Z)$: if it were, we would have $x^2 = r(y^2) = 0$, because the intersection pairing on $H^2(X,\Z)$ is even \cite[Lem.~VIII.15.1(iii)]{barth}.\footnote{The squares here are just taken in the cohomology rings of $X$ or $\tilde X$, with $\Z$ or $\Z/2$ coefficients, as appropriate; we do not use Beauville's quadratic form $q$.}  So $x$ maps to a non-zero class in $\Br(X)$, which must be $\alpha$, and this is enough to prove the proposition.
\end{proof}

\subsection{Induced map on cohomology}

From here on we fix isometries as in \eqref{good_basis} for both $H^2(\tilde X, \Z)$ and $H^2(\tilde Y, \Z)$, and we continue to let $\tau$ denote the involution on both $\tilde X$ and $\tilde Y$, so $\tau^*$ acts on $H^2$ according to \eqref{involution} and acts trivially on $H^0$ and $H^4$.  From Proposition \ref{beauville_prop} we get B-fields on $\tilde X$ and $\tilde Y$, both denoted by $B$, which satisfy $\tau^* B = -B$.

Following Huybrechts and Stellari \cite[\S4]{hs}, the twisted Mukai vector
\[ v^{-B \boxplus B}(\tilde P) \in H^*(\tilde X \times \tilde Y, \Z) \]
induces a map
\[ \phi\colon H^*(\tilde X, \Z) \to H^*(\tilde Y, \Z), \]
which is an isometry with respect to the Mukai pairing, and whose complexification takes $e^B H^{0,2}(\tilde X)$ into $e^B H^{0,2}(\tilde Y)$.

\begin{prop}
$e^{-B} \circ \phi \circ e^B$ commutes with $\tau^*$.
\end{prop}
\begin{proof}
We have
\begin{align*}
(\tau \times \tau)^* v^{-B \boxplus B}(\tilde P)
&= v^{\tau^*(-B) \boxplus \tau^* B}((\tau \times \tau)^* \tilde P) \\
&= v^{B \boxplus (-B)}(\tilde P) \\
&= \operatorname{ch}^{2B \boxplus (-2B)}(\O_{\tilde X \times \tilde Y}) \cdot v^{-B \boxplus B}(\tilde P) \\
&= e^{2B \boxplus (-2B)} \cdot v^{-B \boxplus B}(\tilde P),
\end{align*}
where in the third line we have used \cite[Prop.~1.2(iii)]{hs}, and in the fourth we have used [ibid., Prop.~1.2(ii)].

This implies that
$\tau^* \circ \phi \circ \tau_* = e^{-2B} \circ \phi \circ e^{2B}$,
which can be manipulated to give the desired result.
\end{proof}

\subsection{Integrality}

If we denote the $\tau^*$-invariant and -anti-invariant parts of $H^*$ by $H^*_+$ and $H^*_- = H^2_-$, then we have constructed a Hodge isometry
\[ e^{-B} \circ \phi \circ e^B\colon H^2_-(\tilde X, \Q) \to H^2_-(\tilde Y, \Q). \]
It remains to show that it maps integral classes to integral classes.

To that end, suppose that $x \in H^2(\tilde X,\Z)$ satisfies $\tau^* x = -x$, and write
\[ \phi(e^B x) = (r,c,s) \in H^0(\tilde Y, \Q) \oplus H^2(\tilde Y, \Q) \oplus H^4(\tilde Y, \Q). \]
Then
\[ e^{-B}(r,c,s) = (r, c - rB, s - cB + \tfrac12rB^2) \]
is $\tau^*$-anti-invariant, so $r = 0$, and $s - cB = 0$: that is,
\[ e^{-B} \phi(e^B x) = (0,c,0). \]
So we wish to show that the degree-2 part of $\phi(e^B x)$ is integral.  We have
\[ e^B x = (0,x,y) \in H^0(\tilde X, \Q) \oplus H^2(\tilde X, \Q) \oplus H^4(\tilde X, \Q), \]
where $y = x.B \in \frac12\Z$.  Since $x$ is integral and $y$ is half-integral, we will have proved our main theorem once we prove:

\begin{prop} \label{the_calc}
For any isometry $\phi\colon H^*(\tilde X,\Z) \to H^*(\tilde Y, \Z)$ that commutes with $T := e^B\!\circ\!\tau^*\!\circ\!e^{-B}$, the degree-2 part of $\phi(0,0,1)$ is divisible by~2.
\end{prop}
\begin{proof}
Observe that $T$ is integral: $T = e^{2B} \circ \tau^*$.
\medskip

By Poincar\'e duality, the statement of the proposition is equivalent to
\[ \langle \phi(0,0,1),\, \ell \rangle \equiv 0 \pmod 2 \]
for all $\ell \in H^2(\tilde X,\Z)$, where $\langle-,-\rangle$ denotes the Mukai pairing on $H^*(\tilde X, \Z)$.  Because $T$ is an isometry and $(0,0,1)$ is $T$-invariant, we have
\[ \langle \phi(0,0,1),\, \ell \rangle = \langle \phi(0,0,1),\, \tfrac12(\ell + T\ell) \rangle, \]
so it is enough to show that
\[ \langle \phi(0,0,1),\, \ell + T\ell \rangle \equiv 0 \pmod 4. \]
\pagebreak

\noindent Now our proof will consist of two calculations: \medskip

\noindent \emph{Claim 1}. For any $\ell \in H^2(\tilde X, \Z)$,
\[ \langle \ell + T \ell, \ell + T \ell \rangle \equiv 0 \pmod 4. \]

\noindent \emph{Claim 2}. For any $T$-invariant class $v \in H^*(\tilde X, \Z)$,
\[ \langle (0,0,1),\, v \rangle \equiv \langle v,v \rangle \pmod 4. \]
Observe that this property is preserved by $T$-equivariant isometries, so $\phi(0,0,1)$ has the same property.\footnote{What's going on is that the pairing on the $T$-invariant sublattice of $H^*(\tilde X, \Z)$ is two times an odd unimodular pairing, and $(0,0,1)$ is what's sometimes called a ``characteristic'' or ``parity'' vector.}
\bigskip

To prove the first claim, write
\begin{align*}
\langle \ell + T \ell, \ell + T \ell \rangle
&= \ell^2 + 2 \langle \ell, T\ell \rangle + \langle T\ell, T\ell \rangle \\
&= 2 \ell^2 + 2 \ell.\tau^* \ell.
\end{align*}
Since $\ell^2$ is even, it is enough to show that $\ell.\tau^*\ell$ is even.  Using the basis \eqref{good_basis}, write
\[ \ell = (x,y,z_1,z_2,z_3). \]
Then
\[ \ell.\tau^*\ell = 2x.y + 2z_1.z_2 - {z_3}^2, \]
which is even because ${z_3}^2$ is even.  Thus the first claim is proved.
\bigskip

To prove the second claim, write
\[ v = (r;\ x,y,z_1,z_2,ae+bf;\ s) \in H^0 \oplus H^2 \oplus H^4, \]
where again we use the basis \eqref{good_basis} for $H^2$.  Then
\begin{align*}
Tv &= e^{2B} \tau^* v \\
&= e^{2B} (r;\ y,x,z_2,z_1,-ae-bf;\ s) \\
&= (r;\ y,x,z_2,z_1,(r-a)e + (r-b)f;\ s-a-b+r).
\end{align*}
From $Tv = v$ we find that $x = y$, $z_1 = z_2$, $r = 2a$, and $a = b$.  Thus
\[ v = (2a;\ x,x,z_1,z_1,ae+af;\ s), \]
so
\[ \langle v,v \rangle = 2x^2 + 2{z_1}^2 + 2a^2 - 4as. \]
Since $x^2$ and ${z_1}^2$ are even,
\[ v^2 \equiv 2a^2 \equiv 2a \pmod 4, \]
and moreover
\[ \langle (0,0,1), v \rangle = -2a, \]
so the second claim is proved.
\end{proof}

\newcommand \httpurl [1] {\href{https://#1}{\nolinkurl{#1}}}
\bibliographystyle{plain}
\bibliography{enriques}

\end{document}